\documentclass[11pt]{article}
\pdfoutput=1                            

\usepackage{cite}
\usepackage{amsmath,amssymb,amsfonts,colonequals}
\usepackage{graphicx,xcolor}
\usepackage{textcomp}

\usepackage[T1]{fontenc}				
\usepackage[utf8]{inputenc}				
\usepackage[english]{babel}				
\usepackage[hidelinks]{hyperref}		

\usepackage{fullpage}					
\usepackage{enumitem}					

\graphicspath{{./graphics/}}

\usepackage{algorithm}
\usepackage{algpseudocode}
\usepackage{xpatch}
\def\mathllap{\mathpalette\mathllapinternal}
\def\mathllapinternal#1#2{\llap{$\mathsurround=0pt#1{#2}$}}
\algrenewcommand\algorithmicindent{0.8em}  
\makeatletter
\xpatchcmd{\algorithmic}{\itemsep\z@}{\itemsep=0.5ex plus2pt}{}{}
\makeatother

\usepackage{arydshln}

\usepackage{multirow}

\newcommand{\0}{\mathbf{0}}
\newcommand{\df}{\nabla\! f}
\renewcommand{\L}{\mathcal{L}}
\newcommand{\V}{\mathcal{V}}
\newcommand{\E}{\mathcal{E}}
\newcommand{\W}{W}

\DeclareMathOperator{\cond}{cond}

\usepackage{amsthm}
\newtheorem{thm}{Theorem}

\newtheorem{prop}[thm]{Proposition}
\newcounter{asm}
\newtheorem{assumption}[asm]{Assumption}
\newtheorem{rem}[thm]{Remark}

\def\qed{\rule[0pt]{5pt}{5pt}\par\medskip}
\renewcommand{\qedhere}{\hfill ~\qed}
\renewenvironment{proof}{{\noindent\bf Proof.}}{\qedhere}

\newcommand{\bmat}[1]{\begin{bmatrix}#1\end{bmatrix}}
\newcommand{\squeezearray}{\addtolength{\arraycolsep}{-2pt}}

\newcommand{\tp}{\mathsf{T}}

\renewcommand{\epsilon}{\varepsilon}
\DeclareMathOperator*{\argmin}{\arg\min}

\DeclareMathOperator*{\minimize}{minimize}
\DeclareMathOperator*{\maximize}{maximize}
\newcommand{\defeq}{\colonequals}

\newcommand{\grad}{\nabla\!}

\newcommand{\R}{\mathbb{R}}	

\let\bl\bigl

\let\bbbl\biggl

\let\br\bigr

\let\bbbr\biggr

\newcommand{\norm}[1]{\lVert{#1}\rVert}
\newcommand{\normm}[1]{\bl\lVert{#1}\br\rVert}

\newcommand{\blue}[1]{{\color{blue}#1}}
\renewcommand{\blue}[1]{#1}

\begin{document}
\title{Analysis and Design of First-Order Distributed Optimization Algorithms over Time-Varying Graphs}
\author{Akhil Sundararajan$^{1,2}$ \and Bryan Van Scoy$^{1}$ \and Laurent Lessard$^{1,2}$}
\date{\empty}
\maketitle

\footnotetext[1]{Wisconsin Institute for Discovery, WI~53715, USA.}
\footnotetext[2]{Department of Electrical and Computer Engineering, University of Wisconsin--Madison, WI~53706, USA.
	
	Emails: \{\texttt{asundararaja,vanscoy,laurent.lessard}\}\texttt{@wisc.edu}}

\maketitle


\begin{abstract}
This work concerns the analysis and design of distributed first-order optimization algorithms over time-varying graphs. The goal of such algorithms is to optimize a global function that is the average of local functions using only local computations and communications. Several different algorithms have been proposed that achieve linear convergence to the global optimum when the local functions are strongly convex. We provide a unified analysis that yields the worst-case linear convergence rate as a function of the condition number of the local functions, the spectral gap of the graph, and the parameters of the algorithm.  The framework requires solving a small semidefinite program whose size is fixed; it does not depend on the number of local functions or the dimension of their domain. The result is a computationally efficient method for distributed algorithm analysis that enables the rapid comparison, selection, and tuning of algorithms. Finally, we propose a new algorithm, which we call SVL, that is easily implementable and achieves a faster worst-case convergence rate than all other known algorithms.
\end{abstract}


\section{Introduction}\label{sec:intro}

In distributed optimization, a network of agents, such as computing nodes, robots, or mobile sensors,  work collaboratively to optimize a global objective. Specifically, each agent $i\in\{1,\dots,n\}$ has access to a local function $f_i$ and must minimize the average of all agents' local functions
\begin{equation}\label{eq:distrop-problem}
\min_{x\in \R^d} f(x),
\quad\text{where }f(x) \defeq \frac{1}{n}\sum_{i=1}^n f_i(x),
\end{equation}
\blue{by querying its local gradient $\df_i$, exchanging information with neighboring agents, and performing local computations.

This work aims to study the reliability of distributed optimization algorithms in the presence of a \emph{time-varying} communication graph. Such a scenario could occur if communication links fail due to interference, mobile agents move out of range, or an adversary is jamming communications.}

Distributed optimization is relevant in many application areas.  For example, in large-scale machine learning~\cite{forero2010consensus,johansson2008distributed}, $n$ could represent the number of computing units available for training a large data set. Each $f_i$ then denotes the loss function corresponding to the training examples assigned to  unit $i$. Another example is sensor networks~\cite{rabbat04}, where each sensor may have a limited power budget, communication bandwidth, or sensing capability. \blue{The goal is to aggregate all} local data without having a single point of failure. Other applications include distributed spectrum sensing~\cite{dist_spectrum_sensing} and resource allocation across geographic regions~\cite{dist_power_control}.

\blue{Distributed optimization generalizes both average consensus and centralized optimization, as we now explain.
\paragraph{Consensus} If each agent uses the initial value $x_i^0$ and local objective $f_i(x) = \norm{x-x_i^0}^2$, distributed optimization reduces to \textit{average consensus}~\cite{tsitsiklis,xiaoboyd04}. The unique optimizer of~\eqref{eq:distrop-problem} is then the average of all initial states: $x^\star=\frac{1}{n}\sum_{i=1}^n x_i^0$. Using a \textit{gossip} update of the form $x_i^{k+1} = \sum_{i=1}^n \W_{ij} x_j^k$ where $\W$ is carefully chosen\blue{,} such methods converge exponentially: $\norm{x_i^k -x^\star} \le \rho^k$ with $\rho \in (0,1)$ that depends on $\W$ \cite{FDLA}. This is called a \textit{linear rate} in the optimization community.

\paragraph{Optimization} If $n=1$ or if all $f_i$ are identical, we recover the standard centralized optimization setup. Linear convergence can be guaranteed in certain cases. For example, the gradient descent method $x_i^{k+1} = x_i^k - \alpha \grad f_i(x_i^k)$ achieves linear convergence if $f_i$ is continuously differentiable, smooth, and strongly convex (formally stated in Assumption~\ref{assumption:local_functions})~\cite{NesterovBook}.}

A linear convergence rate for the general case was first achieved by the \textit{exact first-order algorithm} (EXTRA)~\cite{EXTRA}. This algorithm requires storing the previous state in memory:
\begin{subequations}\label{eq:EXTRA}
	\begin{align}
	x_i^1 &= \sum_{j=1}^n \W_{ij}\,x_j^0 - \alpha\, \df_i(x_i^0),\qquad x_i^0\text{ arbitrary}, \\
	\label{eq:EXTRAb}
	x_i^{k+2} &= x_i^{k+1}+\sum_{j=1}^n \W_{ij}\,x_j^{k+1} - \sum_{j=1}^n \widetilde{\W}_{ij}\,x_j^k -\alpha\,\bigl(\df_i(x_i^{k+1})-\df_i(x_i^k)\bigr)
	\end{align}
\end{subequations}
where $\W$ and $\widetilde{\W}$ are gossip matrices that satisfy certain technical conditions and $\alpha$ is sufficiently small.
\blue{Several additional linear-rate algorithms have since been proposed, including: AugDGM~\cite{AugDGM}, DIGing~\cite{DIGing,QuLi}, Exact Diffusion~\cite{ExactDiffusion1,ExactDiffusion2}, NIDS~\cite{NIDS}, and a unified method~\cite{Unification}. Each of these methods have updates similar to~\eqref{eq:EXTRA} in that they require agents to store previous iterates or gradients.}

Although linear convergence rates were obtained for the algorithms above, each algorithm differs in the nature and strength of its convergence analysis guarantees. For example, some works show (non-constructively) the existence of a linear rate~\cite{AB} whereas others provide specific tuning recommendations with associated analytic rate bounds (which may be conservative)~\cite{EXTRA,NIDS}. Numerical simulations are also frequently used~\cite{ABN}, but can be misleading because algorithm performance depends on the graph topology, choice of functions, algorithm initialization, and algorithm tuning.

The present work makes an effort to systematize the analysis and design of distributed optimization algorithms.
We now summarize our main contributions.

\vspace{1mm}\noindent
\textbf{Analysis framework.}
We present a universal analysis framework that \blue{provides an upper bound on the} worst-case linear convergence rate $\rho$ of a wide range of distributed algorithms as a function of the parameters $\kappa$ (local function conditioning) and $\sigma$ (network connectedness).  Our main result, Theorem~\ref{thm:main-result}, is a semidefinite program (SDP) parameterized by
\blue{ $(\kappa,\sigma)$} whose solution yields an upper bound on $\rho$. The SDP has a small fixed size that does not depend on the number of agents $n$ or the dimension of the function domains \blue{and is efficiently solvable}. Our SDP yields robust performance guarantees when the graph is allowed to vary \blue{(even adversarially)} at each iteration. \blue{Fig.~\ref{fig:money_plot} compares the worst-case linear rate $\rho$ for 8 different algorithms.}

\vspace{1mm}\noindent
\textbf{Algorithm design.}
\blue{We present a} new distributed algorithm, which we name SVL (the authors' initials).  SVL is derived by optimizing the SDP from our analysis framework and provides the fastest known convergence rate to date \blue{for this time-varying graph setting}.  \blue{The rate} depends explicitly on~$\kappa$ and~$\sigma$, so no tuning is required if these parameters are known or estimated in advance. \blue{When the graph is well-connected, SVL recovers the performance of gradient descent, which is optimal in this time-varying graph setting.}

\vspace{1mm}\noindent
\textbf{Worst-case examples.}
Although our analysis technique only provides \emph{upper bounds} on the worst-case convergence rate \blue{for} distributed algorithms, we outline a computationally tractable optimization procedure that finds numerically matching lower bounds \blue{by constructing worst-case trajectories, suggesting the bounds found via our analysis technique are tight.}

\begin{rem}[Accelerated rates]
\blue{Distributed algorithms that achieve \textit{accelerated}\cite{QuLi_accelerated,ABN,DHB} or \textit{optimal} \cite{Scaman17} linear rates have also been proposed. It turns out such methods are not guaranteed to achieve acceleration when the graph is time-varying. We discuss this phenomenon in Section~\ref{sec:lower_bounds}, where we derive lower bounds for the time-varying setting.}
\end{rem}

The paper is organized as follows.  We \blue{describe notation and assumptions} in Section~\ref{prelim}. We state and prove our main result for certifying worst-case rate bounds in Section~\ref{main_result}. We present our SVL algorithm and discuss interpretations in Section~\ref{sec:algo_design}. Finally, we demonstrate the tightness of our bounds by generating worst-case trajectories in Section~\ref{numerics}.


\section{Preliminaries}\label{prelim}

\subsection{Notation}\label{notation}

Let $I_n$ be the identity matrix in $\R^{n\times n}$. The symbol $1_n$ denotes the column vector of all ones in $\R^n$. $\Pi \defeq \frac{1}{n}1_n 1_n^\tp$ is the projection matrix onto $1_n$. We will sometimes omit subscripts when dimensions are clear from context.
Unless otherwise indicated, Greek letters denote scalar parameters, lower-case letters denote column vectors, and upper-case letters denote matrices. Exceptions include the scalars $m$ and $L$, which we use in Assumption~\ref{assumption:local_functions} to conform with convention.
The symbol $\otimes$ denotes the Kronecker matrix product.~$\norm{x}$ denotes the standard Euclidean norm of a vector~$x$, and $\norm{A}\defeq \sup_{x\ne 0} \norm{Ax}/\norm{x}$ is the spectral norm of a matrix $A$.
Unless otherwise indicated, subscripts refer to individual agents while superscripts refer to iteration count.  For brevity, we write the symmetric quadratic form $x^\tp Q x$ as $\bmat{\star}^\tp Q x$.

Define the graph $\mathcal{G}\defeq (\V,\E)$ where $\V\defeq \{1,\dots,n\}$ is the set of agents and $\E$ is the set of pairs of agents $(i,j)$ that are connected.  $\L \in \R^{n\times n}$ is a \textit{Laplacian matrix} associated with~$\mathcal{G}$ if $\L 1_n = 0$ and $\L_{ij} = 0$ if $(i,j)\notin \E$.  The \textit{spectral gap} of~$\L$ is defined as the second-smallest eigenvalue magnitude of $\L$.  Since we consider time-varying graphs, we let $\L^k$ denote a Laplacian \blue{matrix associated with} $\mathcal{G}^k$.
\blue{%
We denote a symbol on agent $i$ at iteration $k$ by $x_i^k$ along with its associated fixed point $x_i^\star$. For all such symbols, we denote their aggregation over all agents as
\[
	x^k \defeq \bmat{x_1^k \\ \vdots \\ x_n^k} \quad\text{and}\quad
	x^\star = \bmat{x_1^\star \\ \vdots \\ x_n^\star}.
\]
We denote the associated local and global error coordinates as $\tilde x_i^k \defeq x_i^k - x_i^\star$ and $\tilde x^k \defeq x^k - x^\star$, respectively.
}

\subsection{Function and Graph Assumptions}

\blue{We assume that the local function gradients satisfy the following \emph{sector bound}.}

\begin{assumption}\label{assumption:local_functions}
	Given $0<m\le L$, the the local objective functions $f_i$ are continuously differentiable and each satisfy
	\blue{%
	\[
		\bigl( \df_i(y)-\df_i(y_\text{opt}) - m\,(y-y_\text{opt})\bigr)^\tp \bigl( \df_i(y)-\df_i(y_\text{opt}) - L\,(y-y_\text{opt})\bigr) \le 0
	\]
	for all $y\in\R^d$, where $y_\text{opt}$ satisfies $\sum_{i=1}^n \df_i(y_\text{opt}) = 0$.}
\end{assumption}

\blue{%
\begin{rem}
	One way to satisfy Assumption~\ref{assumption:local_functions} is if the local functions $f_i$ are $L$-Lipschitz continuous and $m$-strongly convex, though in general, Assumption~\ref{assumption:local_functions} is much weaker. 
\end{rem}}

We define the condition ratio as $\kappa\defeq L/m$. This quantity captures how much the curvature of the objective function varies. If $f$ is twice differentiable, $\kappa$ is an upper bound on the condition number of the Hessian~$\nabla^2 f$. \blue{In general, as $\kappa\to\infty$, the functions become poorly conditioned and more difficult to optimize using first-order methods.}

The graph associated with the network of agents can change at each step of the algorithm, so we assume the following about the sequence of graph Laplacian matrices $\{\L^k\}$. 
\begin{assumption}\label{assumption:Laplacian}
	The following properties hold at each step of the algorithm.
		\begin{enumerate}
		\item The graph is connected: there always exists a path between any two nodes in $\mathcal{G}^k$. This implies that the zero eigenvalue of $\L^k$ has a multiplicity of one for all $k$.\label{it:a1} 
		\item The graph is balanced: every node has equal in-degree and out-degree. This means that $1_n^\tp \L^k = 0$ for all $k$.\label{it:a2}
		\item The spectral gap of the time-varying graph is uniformly bounded. In particular, we assume there exists $\sigma \in [0,1)$ such that $\norm{ I-\Pi-\L^k } \leq \sigma$ for all $k$.  Since the spectral radius of a matrix is always upper-bounded by its spectral norm, this implies that $\sigma$ is a uniform bound on the spectral gap of each Laplacian matrix in $\{\L^k\}$.\label{it:a3}
	\end{enumerate}
\end{assumption}

\blue{\begin{rem}\label{rem:B-connected}
The assumption that $\mathcal{G}^k$ must be connected for all $k$ is a strong assumption. Works that consider directed or time varying graphs typically make weaker assumptions, such as a \emph{joint spectrum property} or \emph{$B$-connectedness}~\cite{DIGing}. Nevertheless, our setting (which is equivalent to $B$-connectedness with $B=1$) is still weaker than assuming a constant graph. Indeed, NIDS~\cite{NIDS} converges for any $\sigma$ when the graph is constant, but in Section~\ref{sec:worst-case}, we construct a sequence of graphs that drives NIDS to instability.
\end{rem}}

\subsection{Algorithm Form}

In this paper, we consider the broad class of distributed optimization algorithms that satisfy the algebraic equations
\begin{subequations} \label{alg}
\begin{gather}
	\bmat{ x^{k+1}_i \\ y^k_i \\ z^k_i }
	= \bmat{ A & B_u & B_v \\ C_y & D_{yu} & D_{yv} \\ C_z & D_{zu} & D_{zv} }
	\bmat{ x^k_i \\ u^k_i \\ v^k_i}, \label{alg1} \\
	u^k_i = \df_i (y^k_i), \qquad
	v^k_i = \sum_{j=1}^n\L^k_{ij} z^k_j, \label{alg2}\\
	\sum_{j=1}^n \left( F_x x^k_j + F_u u^k_j \right) = 0. \label{alg3}
\end{gather}
\end{subequations}
Equation~\eqref{alg1} describes how agent $i$'s state $x_i^k$ evolves with iteration $k$. The local gradient $\grad f_i$ is evaluated at $y_i^k$ and the quantity $z_i^k$ is transmitted to neighboring agents in~\eqref{alg2}. Finally, we allow for linear state-input invariants to be enforced in~\eqref{alg3}. Such invariants typically arise from requiring a particular initialization for the algorithm.

The matrices $A$, $D_{yu}$, and $D_{zv}$ are square, and the other matrices have compatible dimensions. The dimension of $A$ is the number of local states on each agent, the dimension of $D_{yu}$ is one, and the dimension of $D_{zv}$ is the number of variables that each agent transmits with neighbors at each iteration.

\blue{%
\begin{rem}[Dimension reduction]
	To simplify notation, we assume the objective function is one-dimensional ($d=1$). We can recover the general $d$ case by replacing each scalar symbol with a $1\times d$ row vector (e.g., $u_i^k\in\R^{1\times d}$) and interpreting each local gradient $\df_i$ as a map from $\R^{1\times d}$ to $\R^{1\times d}$.
\end{rem}

\begin{rem}[Implementation]
	Not all instances of \eqref{alg} are efficiently implementable. For example, if ${D_{yu}\ne 0}$, then $y_i^k$ depends on $u_i^k$, which then depends on $y_i^k$. Such circular dependencies arise naturally in proximal algorithms, where an inner optimization problem must be solved at each iteration. For instance, given a convex differentiable $f$ and parameter $\lambda > 0$, the proximal algorithm
	\[
		x^{k+1} = \mathbf{prox}_{\lambda f}(x^k) \defeq \argmin_{x} \bigl( \lambda f(x) + \tfrac{1}{2}\norm{x-x^k}^2\bigr)
	\]
	satisfies the optimality condition $\lambda \grad f(x^{k+1}) + x^{k+1} - x^k = 0$ and can therefore be expressed in the form of~\eqref{alg} as follows:
	\begin{align*}
		x^{k+1}	&= x^k - \lambda u^k, &
		y^k     &= x^k - \lambda u^k, &
		u^k     &= \df(y^k).
		\end{align*}
	In the forthcoming analysis, we treat implementability and analysis separately. That is, we derive convergence rate bounds for general algorithms of the form~\eqref{alg}, regardless of whether they can be efficiently implemented. However, we note that a sufficient condition for avoiding circular dependencies is if the feedthrough term satisfies
	\begin{align}\label{implementable}
		\bmat{D_{yu} & D_{yv} \\ D_{zu} & D_{zv}} = \bmat{0 & D_{yv} \\ 0 & 0} \quad\text{or}\quad
		\bmat{0 & 0 \\ D_{zu} & 0}.
	\end{align}
\end{rem}

\medskip

\noindent Putting a distributed optimization algorithm into the form of~\eqref{alg} is a straightforward algebraic exercise, which we now demonstrate for two recently proposed algorithms. These algorithms are parameterized by a stepsize $\alpha$ and a gossip matrix $W$. To relate the gossip matrix to the Laplacian matrix, we set $W = I-\mu \mathcal{L}$ for some scalar $\mu\ne 0$. This provides an additional tuning parameter, and is akin to the method of successive overrelaxation used in the numerical solutions of linear systems of equations~\cite{sor_book}.

\paragraph{EXTRA.} The EXTRA algorithm~\eqref{eq:EXTRA} has a state that depends on two previous timesteps. Using the authors' recommendation of $\widetilde{W} = \tfrac{1}{2}(I+W)$ together with $W = I -\mu \L^k$, the equations become
\begin{align*}
	x^1 &= x^0 - \alpha \grad f(x^0) - \mu \L^k x^0, \\
	x^{k+2} &= 2x^{k+1}-x^k -\alpha \left( \grad f(x^{k+1}) - \grad f(x^k) \right) - \mu \L^k \left( x^{k+1} - \tfrac{1}{2} x^k \right).
\end{align*}
Define the state $(x^{k+1},x^k,\grad f(x^k))$. The outputs are now functions of the state:
$y^k \defeq x^{k+1}$ and $z^k \defeq x^{k+1}-\tfrac{1}{2}x^k$. Finally, summing across agents (left-multiplying by $1^\tp$) and using $1^\tp \L^k = 0$, we find that 
$1^\tp \left( x^{k+1} - x^k + \alpha \grad f(x^k) \right)$ is independent of $k$, and identically zero thanks to how $x^1$ is initialized. The parameters that characterize EXTRA are shown below and in Table~\ref{table::algo_params}.
\begin{align*}
\left[\begin{array}{c:c:c}
A & B_u & B_v \\ \hdashline
C_y & D_{yu} & D_{yv} \\ \hdashline
C_z & D_{zu} & D_{zv} \\ \hdashline
F_x & F_u & 
\end{array}\right]
&=
\left[\begin{array}{ccc:c:c}
	2 & -1 & \alpha & -\alpha & -\mu \\
	1 & 0 & 0 & 0 & 0 \\
	0 & 0 & 0 & 1 & 0 \\ \hdashline
	1 & 0 & 0 & 0 & 0 \\ \hdashline
	1 & -\tfrac{1}{2} & 0 & 0 & 0 \\ \hdashline
	1 & -1 & \alpha & 0 &
\end{array}\right].
\end{align*}

\paragraph{DIGing.} The DIGing algorithm~\cite{DIGing,QuLi}, is an example of a \emph{gradient tracking} algorithm. It begins with an arbitrary $x^0$ and has two update equations:
\begin{align*}
s^0 &= \grad f(x^0), \\
x^{k+1} &= W x^k -\alpha s^k, \\
s^{k+1} &= \widetilde{W}s^k  + \grad f (x^{k+1}) - \grad f(x^k).
\end{align*}
Using the authors' recommendation of $\widetilde{W} = W$, defining $W = I -\mu \L^k$ as before, and defining the state as $(x^{k},s^k,\grad f(x^k))$, we find that the output is $y^k \defeq x^{k+1}$, two quantities must be communicated between agents, $z^k \defeq (x^k, s^k)$, and the invariant is $1^\tp (s^k - \df(x^k)) = 0$.
The parameters that characterize DIGing are shown in Table~\ref{table::algo_params}.

A similar derivation can be applied to a variety of algorithms. Table~\ref{table::algo_params} summarizes the parameterizations for 8 recently proposed algorithms.
}

\begin{table*}[htb!]
	\caption{Algorithm parameters in the form of~\eqref{alg} for a variety of different distributed optimization algorithms. Algorithms can be tuned by choosing stepsize and overrelaxation parameters $\alpha$ and $\mu$, respectively. Algorithms are organized based on how many internal states they have (columns) and how many variables must be communicated in each iteration (block rows).}
	\vspace{1mm}
	\label{table::algo_params}
	\centering
	\renewcommand{\arraystretch}{1.1}%
	\setlength{\fboxsep}{1pt}
	\setlength{\fboxrule}{0pt}
	\begin{tabular}{c|p{2.7cm}c|p{2.45cm}c}
		& \multicolumn{2}{c|}{Algorithms with 2 states} & \multicolumn{2}{c}{Algorithms with 3 states} \\
	\hline
	\parbox[t]{2mm}{\multirow{2}{*}{\rotatebox[origin=c]{90}{1 communicated variable}}}
	& \parbox{2.7cm}{\small \textbf{SVL template}\\
	\textbf{(present work)}\\
	See Section~\ref{sec:algo_design}\\
	for derivation\\
	of $(\alpha,\beta,\gamma,\delta)$} & $\squeezearray\left[\begin{array}{cc:c:c}
		1 & \beta & -\alpha & -\gamma \\
		0 & 1 & 0 & -1 \\ \hdashline
		1 & 0 & 0 & -\delta \\ \hdashline
		1 & 0 & 0 & 0 \\ \hdashline
		0 & 1 & 0 & 
	\end{array}\right]$ &
	EXTRA~\cite{EXTRA} & \fbox{$\squeezearray\left[\begin{array}{ccc:c:c}
		2 & -1 & \alpha & -\alpha & -\mu \\
		1 & 0 & 0 & 0 & 0 \\
		0 & 0 & 0 & 1 & 0 \\ \hdashline
		1 & 0 & 0 & 0 & 0 \\ \hdashline
		1 & -\tfrac{1}{2} & 0 & 0 & 0 \\ \hdashline
		1 & -1 & \alpha & 0 &
	\end{array}\right]$} \\
	& \parbox{2.7cm}{Exact Diffusion\\
	(ExDIFF)\\
	\cite{ExactDiffusion1,ExactDiffusion2}} & $\squeezearray\left[\begin{array}{cc:c:c}
		2 & -1 & -\alpha & -\mu \\
		1 & 0 & -\alpha & -\tfrac12\mu \\ \hdashline
		1 & 0 & -\tfrac12 \mu & 0 \\ \hdashline
		1 & 0 & 0 & 0 \\ \hdashline
		1 & -1 & 0 & 
	\end{array}\right]$ &
	NIDS~\cite{NIDS} & \fbox{$\squeezearray\left[\begin{array}{ccc:c:c}
		2 & -1 & \alpha & -\alpha & -\mu \\
		1 & 0 & 0 & 0 & 0 \\
		0 & 0 & 0 & 1 & 0 \\ \hdashline
		1 & 0 & 0 & 0 & 0 \\ \hdashline
		1 & -\tfrac{1}{2} & \tfrac{\alpha}{2} & -\tfrac{\alpha}{2} & 0 \\ \hdashline
		1 & -1 & \alpha & 0 &
	\end{array}\right]$} \\
	\hline
	\parbox[t]{2mm}{\multirow{2}{*}{\rotatebox[origin=c]{90}{2 communicated variables}}}
	& \parbox{2.6cm}{Unified DIGing\\
	(uDIG)~\cite{Unification}} & $\squeezearray\left[\begin{array}{cc:c:cc}
		1 & -\alpha & -\alpha & -\mu & 0 \\
		0 & 1 & 0 & 0 & -\mu \\ \hdashline
		1 & 0 & 0 & 0 & 0 \\ \hdashline
		1 & 0 & 0 & 0 & 0 \\
		-\tfrac{L+m}{2} & 1 & 1 & 0 & 0\\ \hdashline
		0 & 1 & 0 & & 
	\end{array}\right]$ &
	DIGing~\cite{DIGing,QuLi} & \fbox{$\squeezearray\left[\begin{array}{ccc:c:cc}
		1 & -\alpha & 0 & 0 & -\mu & 0\\
		0 & 1 & -1  & 1 & 0 & -\mu \\
		0 & 0 & 0   & 1 & 0 & 0 \\ \hdashline
		1 & -\alpha & 0   & 0 & -\mu & 0 \\ \hdashline
		1 & 0 & 0 & 0 & 0 & 0 \\
		0 & 1 & 0 & 0 & 0 & 0 \\ \hdashline
		0 & 1 & -1 & 0 & &
	\end{array}\right]$} \\
	& \parbox{2.7cm}{Unified EXTRA\\
	(uEXTRA)~\cite{Unification}} & $\squeezearray\left[\begin{array}{cc:c:cc}
		1 & -\alpha & -\alpha & -\mu & 0 \\
		0 & 1 & 0 & 0 & -\mu \\ \hdashline
		1 & 0 & 0 & 0 & 0 \\ \hdashline
		1 & 0 & 0 & 0 & 0 \\
		-L & 1 & 1 & L\mu & 0\\ \hdashline
		0 & 1 & 0 & & 
	\end{array}\right]$ &
	AugDGM~\cite{AugDGM} & \fbox{$\squeezearray\left[\begin{array}{ccc:c:cc}
		1 & -\alpha & 0 & 0 & -\mu & \alpha\mu\\
		0 & 1 & -1      & 1 & 0 & -\mu \\
		0 & 0 & 0       & 1 & 0 & 0 \\ \hdashline
		1 & -\alpha & 0   & 0 & -\mu & \alpha\mu \\ \hdashline
		1 & 0 & 0 & 0 & 0 & 0 \\
		0 & 1 & 0 & 0 & 0 & 0 \\ \hdashline
		0 & 1 & -1 & 0 & &
	\end{array}\right]$} \\
	\hline
	\end{tabular}
\end{table*}

\subsection{Existence of a Fixed Point}

\blue{%
Not all instances of algorithm~\eqref{alg} solve the distributed optimization problem~\eqref{eq:distrop-problem}. For an algorithm to be valid, (i) there must exist a fixed point corresponding to the optimal solution, and (ii) the iterates must converge to the fixed point. We address convergence to a fixed point in our main result of Section~\ref{main_result}. In this section, however, we provide simple conditions for verifying the existence of such a fixed point.

A distributed algorithm of the form~\eqref{alg} has a fixed point $(x^\star, y^\star, z^\star, u^\star, v^\star)$ corresponding to the optimal solution of~\eqref{eq:distrop-problem} for all functions satisfying Assumption~\ref{assumption:local_functions} and all graphs satisfying Assumption~\ref{assumption:Laplacian} if the following conditions hold.
\begin{subequations}\label{eq:fixedpoint_conditions}
\begin{itemize}
	\item \textbf{Consensus and Optimality:} All agents must achieve consensus on the point at which the gradient is evaluated, and the point must be a stationary (first-order optimal) point of $f$. This means that the fixed point must satisfy $y_1^\star = \ldots = y_n^\star$ and $u_1^\star + \dots + u_n^\star = 0$, or in vector form,
	\begin{equation}\label{eq:fixedpoint_consensus_optimality}
		(I-\Pi)\, y^\star = 0 \quad\text{and}\quad
		1^\tp u^\star = 0.
	\end{equation}
	\item \textbf{Robustness to Graph:} The fixed point must not depend on the sequence of graphs $\{\L^k\}$, so $z_1^\star = \ldots = z_n^\star$ and $v_1^\star = \dots = v_n^\star = 0$, or in vector form,
	\begin{equation}
		 (I-\Pi)\, z^\star = 0 \quad\text{and}\quad
		 v^\star = 0.
	\end{equation}
	\item \textbf{Robustness to Functions:} The fixed point must satisfy $y_1^\star = \ldots = y_n^\star = y_\text{opt}$ and $u_i^\star = \df_i(y_\text{opt})$, where $y_\text{opt}$ is the optimizer of~\eqref{eq:distrop-problem}. For these to hold for any objective function $f$, we need
	\begin{equation}
		1^\tp y^\star \text{ and } (I-\Pi)\, u^\star \text{ unconstrained}.
	\end{equation}
\end{itemize}
\end{subequations}
The following proposition characterizes algorithms with such a fixed point, which we prove in Appendix~\ref{app:fixedpoint}.

\begin{prop}[Existence of fixed point]\label{prop:fixedpoint}
	An algorithm of the form~\eqref{alg} has a fixed point $(x^\star,y^\star,z^\star,u^\star,v^\star)$ that satisfies the conditions in~\eqref{eq:fixedpoint_conditions} if and only if
	\begin{subequations}\label{eq:fixedpoint}
	\begin{gather}
		\label{eq:fixedpoint1}
		\mathrm{null}(A-I) \cap \mathrm{row}(C_y) \cap \mathrm{null}(F_x) \neq \{0\}\\
		\label{eq:fixedpoint2}
		\text{and}\quad
		\bmat{ B_u \\ D_{yu} \\ D_{zu} } \in \mathrm{col}\!\left(\bmat{ A-I \\ C_y  \\ C_z }\right).
	\end{gather}
	\end{subequations}
\end{prop}

Here, ``null'', ``col'', and ``row'' denote the nullspace, column space, and row space, respectively.
Both EXTRA and DIGing as derived above satisfy the conditions in~\eqref{eq:fixedpoint} and therefore have a fixed point corresponding to the optimal solution of~\eqref{eq:distrop-problem}.
}

\begin{rem}\label{rem:convergence}
	Proposition~\ref{prop:fixedpoint} guarantees that any instance of algorithm~\eqref{alg} \blue{satisfying~\eqref{eq:fixedpoint_conditions}} has a desirable fixed point in the presence of a time-varying graph; all agents agree on a common stationary point of~\eqref{eq:distrop-problem}. However, Proposition~\ref{prop:fixedpoint} does not ensure that the algorithm necessarily converges to this fixed point, nor does it characterize the rate of convergence. These questions will be explored in Section~\ref{main_result}. 
\end{rem}

\subsection{Lower Bounds on Worst-Case Convergence Rates}\label{sec:lower_bounds}

\blue{%
We now construct simple lower bounds on the worst-case asymptotic convergence rate of the iterates for any valid algorithm of the form~\eqref{alg}. We do so by separately considering the two specific instances discussed in Section~\ref{sec:intro}
\paragraph{Consensus} Consider the scalar local quadratic functions $f_i(y) = \tfrac{L}{2}\,(y-r_i)^2$. Then Assumption~\ref{assumption:local_functions} holds with $m=L$ and $y_\text{opt}=\tfrac{1}{n} \sum_{i=1}^n r_i$.
\paragraph{Optimization} Consider the case $n=1$. For the graph to satisfy Assumption~\ref{assumption:Laplacian}, the Laplacian matrix must be $\L^k=0$, which has spectral gap $\sigma=0$.

In both cases above, the algorithm reduces to a linear system in feedback with sector-bounded nonlinearity: in the sector $(1-\sigma,1+\sigma)$ for consensus and $(m,L)$ for optimization. Further, the linear part of the system is strictly proper (since the algorithm is implementable) and must contain an integrator (due to the fixed-point conditions). Then using the lower bound for such systems in~\cite{lowerbound}, we obtain the following.

\begin{prop}
  There does \emph{not} exist an algorithm of the form~\eqref{alg} that satisfies the implementability conditions~\eqref{implementable} and fixed-point conditions~\eqref{eq:fixedpoint} and such that, for all objective functions and Laplacian matrices satisfying Assumptions~\ref{assumption:local_functions} and~\ref{assumption:Laplacian}, there exists a constant $c>0$ such that the bound $\|x_i^k-y_\text{opt}\|\le c\,\rho_\text{lb}^k$ holds for all agents $i\in\{1,\ldots,n\}$ and all iterations $k\ge 0$, where $\rho_\text{lb} = \max\bigl\{\tfrac{\kappa-1}{\kappa+1}, \, \sigma\bigr\}$.
\end{prop}
\begin{rem}[Accelerated rates]\label{rem:no_free_lunch}
	These lower bounds, which are achieved by ordinary gradient descent, imply that accelerated algorithms such as the recently proposed SSDA~\cite{Scaman17} or distributed versions of heavy-ball~\cite{DHB} or Nesterov acceleration~\cite{QuLi_accelerated,ABN} do not in fact achieve accelerated rates in the worst case in our time-varying setting.
\end{rem}
}


\section{Main Result}\label{main_result}
Our main theorem, Theorem~\ref{thm:main-result}, consists of a small convex semidefinite program (SDP) whose feasibility guarantees the linear convergence of a distributed algorithm in the form of~\blue{\eqref{alg}}. The algorithm parameters, problem data $(\kappa, \sigma)$, and candidate linear rate $\rho$ all appear as parameters in the SDP. Furthermore, the SDP has a fixed size that does not depend on $n$ (the number of agents) or $d$ (the dimension of the domain of $f$) and can thus be efficiently solved using a variety of established solvers.

\blue{%
\begin{thm}[Analysis result]\label{thm:main-result}
	Consider the distributed optimization problem~\eqref{eq:distrop-problem} solved using algorithm~\eqref{alg}. Suppose Assumptions~\ref{assumption:local_functions} and \ref{assumption:Laplacian} hold and further assume the algorithm satisfies the fixed point conditions~\eqref{eq:fixedpoint}. Define the matrices
	\begin{gather*}
		M_0 \defeq \bmat{-2mL & L+m \\ L+m & -2} \quad\text{and}\quad
		M_1 \defeq \bmat{\sigma^2-1 & 1 \\ 1 & -1}.
	\end{gather*}
	Let $\Psi$ be a matrix whose columns form a basis for the nullspace of $\bmat{ F_x & F_u}$. If there exist $P \succ 0$, $Q\succ 0$, and $R\succeq 0$ of appropriate sizes such that	
	\begin{subequations}\label{eq:sdp-main}
		\begin{align}
		\Psi^\tp
		\left[\begin{array}{cc}
			A & B_u\\ I & 0 \\ \hdashline
			C_y & D_{yu}\\ 0 & I
		\end{array}\right]^\tp
		\left[\begin{array}{cc:c}
			P & 0 & 0 \\
			0 & -\rho^2 P & 0 \\ \hdashline
			0 & 0 & M_0
		\end{array}\right]
		\left[\begin{array}{cc}
			A & B_u\\ I & 0 \\ \hdashline
			C_y & D_{yu}\\ 0 & I
		\end{array}\right]
		\Psi &\preceq 0 \label{eq:lb-rho} \\
		\left[\begin{array}{ccc}
			A & B_u & B_v \\
			I & 0 & 0 \\ \hdashline
			C_y & D_{yu} & D_{yv} \\
			0 & I & 0 \\ \hdashline
			C_z & D_{zu} & D_{zv} \\
			0 & 0 & I
		\end{array}\right]^\tp
		\left[\begin{array}{cc:c:c}
			Q & 0 & 0 & 0 \\
			0 & -\rho^2 Q & 0 & 0 \\ \hdashline
			0 & 0 & M_0 & 0 \\ \hdashline
			0 & 0 & 0 & M_1\otimes R
		\end{array}\right]
		\left[\begin{array}{ccc}
			A & B_u & B_v \\
			I & 0 & 0 \\ \hdashline
			C_y & D_{yu} & D_{yv} \\
			0 & I & 0 \\ \hdashline
			C_z & D_{zu} & D_{zv} \\
			0 & 0 & I
		\end{array}\right]
			&\preceq 0
		\label{eq:lmi-dis}
	\end{align}
	\end{subequations}
	then there exists a constant $c>0$ independent of $i$ and $k$ such that for all agents $i\in\{1,\dots,n\}$ and all iterations $k\ge 0$,
	\begin{equation}\label{eq:bound}
		\norm{ x_i^k - x_i^\star } \le c\,\rho^k
	\end{equation}
	for some fixed point $(x_i^\star, y_i^\star, z_i^\star, u_i^\star, v_i^\star)$ that satisfies~\eqref{eq:fixedpoint_conditions}.
\end{thm}
}

\bigskip

\blue{%
For fixed algorithm parameters $A,B_u,B_v,C_y,C_z,D_{yu}$, $D_{yv},D_{zu},D_{zv},F_x,F_u$, function parameters $m$ and $L$, graph parameter $\sigma$, and candidate rate $\rho$, the SDP~\eqref{eq:sdp-main} is a linear matrix inequality (LMI) in the variables $(P,Q,R)$, and therefore convex. Indeed, \eqref{eq:lb-rho} and \eqref{eq:lmi-dis} are decoupled and their feasibility may be checked separately.}
To find the best (smallest) upper bound, we observe that feasibility of~\eqref{eq:sdp-main} for some $\rho_0$ implies feasibility for all $\rho \ge \rho_0$. A bisection search on $\rho$ is then guaranteed to find the minimal $\rho$, even though~\eqref{eq:sdp-main} is not jointly convex in $(P,Q,R,\rho)$. \blue{While our result is only a sufficient condition for convergence, we provide empirical evidence in Section~\ref{sec:worst-case} that suggests that it is in fact tight.}

\blue{%
\begin{rem}\label{rem:bounds}
	Our main theorem provides conditions under which the \textit{state} converges to a fixed point linearly with rate~$\rho$. However, when the algorithm also satisfies the conditions in~\eqref{implementable} for being efficiently implementable, then under the conditions of Theorem~\ref{thm:main-result}, there exist constants $c_u$, $c_v$, $c_y$, and $c_z$ such that for all agents $i$ and all iterations $k$,
	\begin{align*}
		\|u_i^k - u_i^\star\| &\le c_u\,\rho^k, &
		\|y_i^k - y_i^\star\| &\le c_y\,\rho^k, &
		\|v_i^k - v_i^\star\| &\le c_v\,\rho^k, &
		\|z_i^k - z_i^\star\| &\le c_z\,\rho^k,
	\end{align*}
	for some fixed point $(x_i^\star, y_i^\star, z_i^\star, u_i^\star, v_i^\star)$ that satisfies~\eqref{eq:fixedpoint_conditions}. In particular, the output sequence $y_i^k$ of each agent converges to the optimizer $y_\text{opt}$ of~\eqref{eq:distrop-problem} linearly with rate $\rho$.
\end{rem}}

The core idea behind Theorem~\ref{thm:main-result} is to posit a quadratic Lyapunov candidate of the form
\blue{%
\begin{align}\label{eq:V}
V^k\defeq (x^k - x^\star)^\tp \bigl( \Pi\otimes P + (I-\Pi)\otimes Q\bigr) (x^k - x^\star)
\end{align}
for some appropriate choice of $P,Q \succ 0$.}
Feasibility of~\eqref{eq:sdp-main} can be shown to imply $V^{k+1}\leq \rho^2 V^k$, which ensures linear convergence of the distributed optimization algorithm when~$\rho<1$.  A preliminary (and less concise) version of Theorem~\ref{thm:main-result} appeared in~\cite{sundararajan_allerton}. \blue{The proof of Theorem~\ref{thm:main-result} is given in Appendix~\ref{proof_main}.}


\section{Algorithm Design}\label{sec:algo_design}

\blue{We now use Theorem~\ref{thm:main-result} to design a distributed optimization algorithm, which we name SVL. Our guiding principle is to seek the fastest possible rate bound guarantee while keeping the algorithm as simple as possible. Therefore, we seek an algorithm with two states that only requires one state to be communicated at every timestep. Inspired by our previous work in which we developed a canonical form for distributed algorithms over time-invariant graphs~\cite{canform}, we restrict our search to algorithms of the form~\eqref{alg} with
\begin{align}\label{eq:SVL-template}
	\left[\begin{array}{c:c:c}
	A & B_u & B_v \\ \hdashline
	C_y & D_{yu} & D_{yv} \\ \hdashline
	C_z & D_{zu} & D_{zv} \\ \hdashline
	F_x & F_u &
	\end{array}\right]
	&=
	\left[\begin{array}{cc:c:c}
		1 & \beta & -\alpha & -\gamma \\
		0 & 1 & 0 & -1 \\ \hdashline
		1 & 0 & 0 & -\delta \\ \hdashline
		1 & 0 & 0 & 0 \\ \hdashline
		0 & 1 & 0 & 
	\end{array}\right].
\end{align}
As long as $\beta \ne 0$, this algorithm satisfies the fixed point conditions of Proposition~\ref{prop:fixedpoint}. Moreover, the update equations satisfy~\eqref{implementable} and therefore do not contain circular dependencies, so we can implement the algorithm in a straightforward fashion as in Algorithm~\ref{alg:canonical}. To motivate the structure of our algorithm, we show how it corresponds to an inexact version of the alternating direction method of multipliers (ADMM), as well as how it reduces to well-known consensus and optimization algorithms in special cases. But first, we show how to use the SDP~\eqref{eq:sdp-main} to choose the algorithm parameters.

\begin{algorithm}[htb]
	\small\caption{\blue{(template for the SVL algorithm)}}\label{alg:canonical}
	\begin{algorithmic}
		\blue{
		\State\textbf{Initialization:}~Let $\L^k \in\R^{n\times n}$ be a Laplacian matrix. \blue{Agents $i\in\{1,\ldots,n\}$ choose initial local state $x_i^0 \in \R^d$ arbitrarily and $w_i^0\in\R^d$ such that $\sum_{i=1}^{n}w_i^0=0$ (e.g. $w_i^0=0$).}
		\For {iteration $k=0,1,2,\ldots$}
		\For {agent $i\in\{1,\ldots,n\}$}
		\State \textbf{Local communication}
		\State $~~\hphantom{w_i^{k+1}}\mathllap{v_{i}^k} = \sum_{j=1}^n \L^k_{ij}\,x_j^k \hfill \text{(C.1)}$
		\State \textbf{Local gradient computation}
		\State $~~\hphantom{w_i^{k+1}}\mathllap{y_i^k} = x_i^k - \delta\,v_{i}^k \hfill \text{(C.2)}$
		\State $~~\hphantom{w_i^{k+1}}\mathllap{u_i^k} = \df_i(y_i^k) \hfill \text{(C.3)}$    
		\State \textbf{Local state update}
		\State $~~\hphantom{w_i^{k+1}}\mathllap{x_i^{k+1}} = x_i^k + \beta\,w_i^k - \alpha\,u_i^k - \gamma\, v_{i}^k \hfill \text{(C.4)}$
		\State $~~\hphantom{w_i^{k+1}}\mathllap{w_i^{k+1}} = w_i^k - v_{i}^k \hfill \text{(C.5)}$
		\EndFor
		\EndFor
		}
	\end{algorithmic}
\end{algorithm}
}

\subsection{Choosing the Algorithm Parameters}

The problem of minimizing the worst-case convergence rate~$\rho$ over the algorithm parameters $(\alpha,\beta,\gamma,\delta)$ and SDP solution \blue{$(P,Q,R)$} subject to the SDP being feasible is difficult due to the nonlinear matrix \blue{inequalities~\eqref{eq:sdp-main}}. Instead, we show that for a particular choice of $(\alpha,\gamma,\delta)$, the remaining parameters $(\beta,\rho)$ can be chosen such that the \blue{SDP is feasible, where the} matrix in~\eqref{eq:lmi-dis} is rank one. 
We have performed extensive numerical optimizations of the SDP, suggesting that the optimal parameters do in fact have this structure. We now state our main design result, which describes \blue{the convergence rate of} the SVL algorithm.  We prove the result in Appendix~\ref{sec:SVL-proof}.

\begin{thm}[SVL]\label{thm:SVL}
  Consider applying Algorithm~\ref{alg:canonical} to the distributed optimization problem~\eqref{eq:distrop-problem}, and suppose Assumptions~\ref{assumption:local_functions} and \ref{assumption:Laplacian} hold \blue{with $0<m<L$ and $0\le\sigma < 1$}. Define $\eta\defeq 1+\rho-\kappa\,(1-\rho)$ and choose the parameters
  \begin{align}\label{eq:paramopt}
  \alpha &= \frac{1-\rho}{m}, &
  \gamma &= 1+\beta, &
  \delta &= 1,
  \end{align}
  where $\beta$ and $\blue{\rho\in\bigl[\tfrac{L-m}{L+m},1\bigr)}$ satisfy the constraints
  \begin{subequations}\label{eq:sol-opt}
  \begin{align}
  \bigl(2\beta-(1-\rho)(\kappa+1)\bigr)(\beta-1+\rho^2) &< 0, \label{eq:beta-constraint}\\
  \rho^2\,\bbbl(\frac{\beta-1+\rho^2}{\beta-1+\rho}\bbbr)
    \bbbl(\frac{2-\eta-2\beta}{2\rho^2\beta - (1-\rho^2)\eta}\bbbr) 
    \bbbl(\frac{(2\rho^2+\eta)\beta - (1-\rho^2)\eta}{(1+\rho)(\eta-2\eta\rho+2\rho^2)-(2\rho^2+\eta)\beta}\bbbr) &= \sigma^2. \label{eq:sigopt}
  \end{align}
  \end{subequations}
  Then there exists a constant $c>0$ independent of $i$ and $k$ such that for all agents $i\in\{1,\dots,n\}$ and all iterations $k\ge 0$,
$\norm{ y_i^k - y_\text{opt} } \le c\,\rho^k $
  where $y_\text{opt}\in\R^d$ is the optimizer of~\eqref{eq:distrop-problem}.
\end{thm}

Theorem~\ref{thm:SVL} provides conditions on parameters $(\alpha,\beta,\gamma,\delta)$ of Algorithm~\ref{alg:canonical} such that the algorithm converges with rate at least $\rho$. The theorem, however, does not address the problem of optimizing the convergence rate since $\beta$ and $\rho$ must only be chosen to satisfy the constraints~\eqref{eq:sol-opt}. This is because the optimal parameters do not admit a closed-form solution for the convergence rate $\rho$ as a function of the spectral gap $\sigma$ and function parameters $m$ and $L$. However, we now provide a systematic method for computing the optimal parameters.

The parameters must satisfy~\eqref{eq:sigopt}, but this equation does not have a closed-form solution for $\rho$. Instead, we consider fixing the \blue{rate~$\rho$} and maximizing the corresponding spectral gap. We can then choose $\beta$ to maximize $\sigma^2$ in~\eqref{eq:sigopt}. Setting the derivative equal to zero, we find that the value of $\beta$ which maximizes $\sigma^2$ for a fixed convergence rate $\rho$ satisfies
\[
\frac{\textrm{d}\sigma^2}{\textrm{d}\beta}=0 \quad\implies\quad 
\bigl(\beta\bigl(1-\kappa+2\rho(1+\rho)\bigr)-\eta(1-\rho^2)\bigr)
\bigl(s_0 + s_1\beta + s_2\beta^2 + s_3\beta^3\bigr)=0,
\]
where the coefficients $s_i$ are given by
\begin{align*}
s_0 &\defeq \eta\,\bigl(1-\rho^2\bigr)^2 \bigl(\eta-(3-\eta)\eta\rho+2(1-\eta)\rho^2+2\rho^3\bigr), \\
s_1 &\defeq -\bigl(1-\rho^2\bigr) \Bigl(\eta^3\rho+4\rho^5-2\eta\rho^2(2\rho^2+\rho-3)+\eta^2\,\bigl(4 \rho^3-4 \rho^2-6 \rho+3\bigr)\Bigr), \\
s_2 &\defeq 3\eta(1-\rho)^2 (1+\rho) (2\rho^2+\eta), \\
s_3 &\defeq (2\rho^2+\eta) (2\rho^3-\eta).
\end{align*}
Solving the first factor for~$\beta$, we find that it does not satisfy the inequality~\eqref{eq:beta-constraint} and is therefore not a valid solution. The optimal $\beta$ must then make the second factor zero. Therefore, we can do a bisection search over $\rho$, where at each iteration of the bisection search we solve the cubic equation
\begin{align}\label{eq:cubic}
s_0 + s_1\beta + s_2\beta^2 + s_3\beta^3 = 0
\end{align}
to find the unique real solution $\beta$ that satisfies~\eqref{eq:beta-constraint}. Substituting this value for $\beta$ into~\eqref{eq:sigopt} we can solve for~$\sigma$. If this value is less than $\sigma$, we increase~$\rho$; otherwise, we decrease~$\rho$. We then repeat this procedure until~$\sigma$ is sufficiently close to the spectral gap. We summarize this procedure for finding the parameters $\beta$ and $\rho$ that optimize the worst-case convergence rate in Algorithm~\ref{alg:paramopt}; we refer to Algorithm~\ref{alg:canonical} using these parameters along with those in~\eqref{eq:paramopt} as SVL.

\begin{algorithm}[htb]
	\small\caption{(computing the SVL parameters)}\label{alg:paramopt}
	\begin{algorithmic}
		\State\textbf{Initialization:}~Let $0<m<L$, $0\le\sigma<1$, and $\epsilon>0$. Define $\kappa\defeq L/m$. Set $\rho_1=0$ and $\rho_2=1$.
		\While {$\rho_2-\rho_1 > \epsilon$}
		\State $\rho = (\rho_1+\rho_2)/2$
		\State Let $\beta$ be the unique real solution to~\eqref{eq:cubic} that satisfies~\eqref{eq:beta-constraint}.
		\State Using this value of $\beta$, let $\hat{\sigma}$ denote the solution to~\eqref{eq:sigopt}.
		\If {$\hat{\sigma} < \sigma$}
		  \State $\rho_1 = \rho$
		\Else
		  \State $\rho_2 = \rho$
		\EndIf
		\EndWhile \\
		\Return $\rho,\beta$
	\end{algorithmic}
\end{algorithm}

Using this procedure for computing the worst-case convergence rate of SVL, Fig.~\ref{fig:optalg} displays~$\rho$ as a function of the spectral gap~$\sigma$ and the centralized gradient rate $\tfrac{\kappa-1}{\kappa+1}$. One of the remarkable aspects of the SVL algorithm is that it actually achieves the same worst-case convergence rate as \emph{centralized} gradient descent if the spectral gap is sufficiently small. In this case, there is sufficient mixing among the agents so that the convergence rate is limited by the difficulty of the optimization problem and not the problem of having agents agree on the solution (i.e., consensus). This corresponds to the horizontal lines for small values of $\sigma$ in the \blue{top} panel of Fig.~\ref{fig:optalg}. Viewed another way, the convergence rate is limited by the difficulty of the optimization problem when the problem is ill-conditioned (i.e., $\kappa$ is large), which corresponds to the curves approaching the straight line at $\rho=\tfrac{\kappa-1}{\kappa+1}$ in the \blue{bottom} panel of Fig.~\ref{fig:optalg}.

\blue{%
\begin{rem}[Optimality]
  We conjecture that the SVL parameters $(\alpha,\beta,\gamma,\delta)$ produce the fastest worst-case convergence rate over all algorithms in the form of Algorithm~\ref{alg:canonical} that is certifiable using Theorem~\ref{thm:main-result}. However, we make no formal claims of optimality of the SVL algorithm in this paper.
\end{rem}
}

\begin{figure}[tbh]
	\centering
	\includegraphics[scale=0.90]{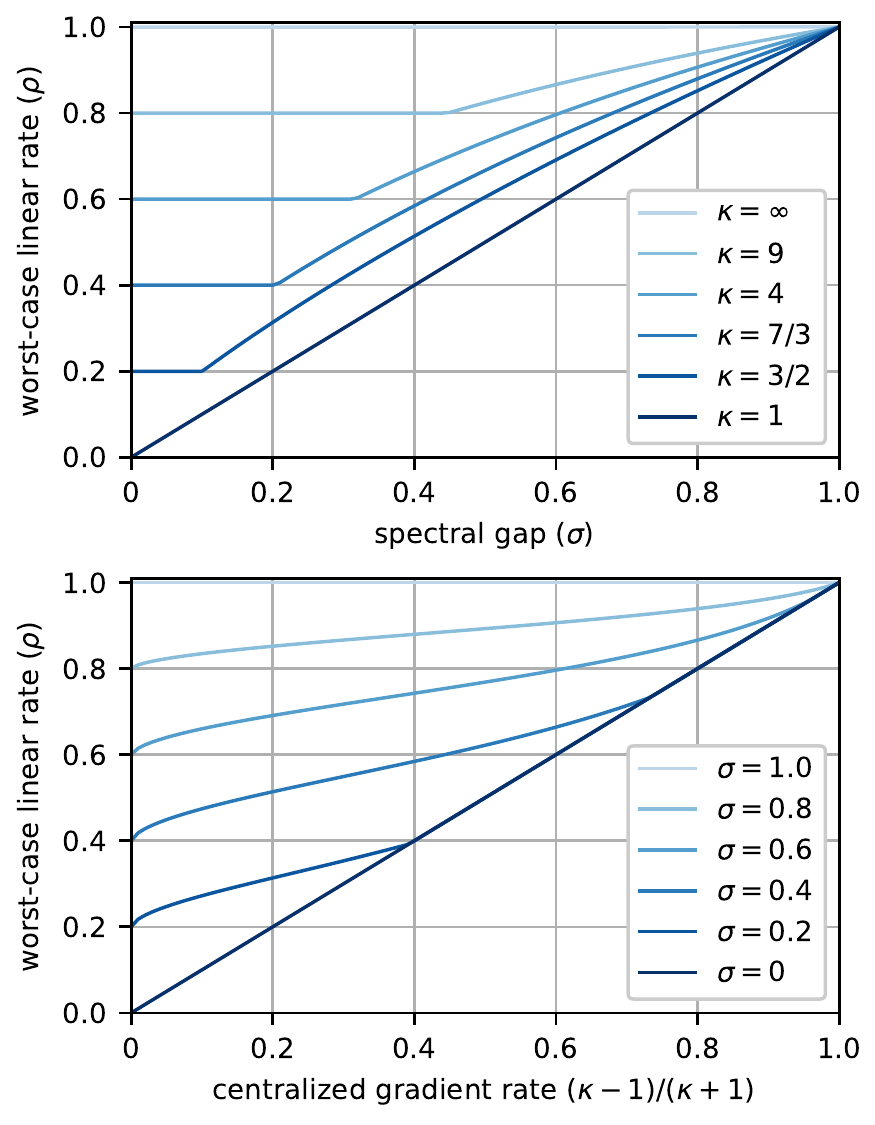}
	\caption{Worst-case \blue{linear rate $\rho$} of SVL in Theorem~\ref{thm:SVL} \blue{as a function of $\kappa$ and $\sigma$. Top plot: as $\kappa\to 1$ (quadratic objective), we obtain $\rho = \sigma$ (optimal linear consensus rate). Bottom plot: as $\sigma\to 0$ (fully connected graph), we obtain $\rho = \tfrac{\kappa-1}{\kappa+1}$ (optimal centralized gradient rate).} }\label{fig:optalg}
\end{figure}

\subsection{Interpretation of SVL as Inexact ADMM}\label{sec:admm}

To motivate the structure of SVL, we show how SVL can be interpreted as an inexact version of the alternating direction method of multipliers (ADMM). Using the formulation in~\cite[Section 7.1]{ADMM}, the problem~\eqref{eq:distrop-problem} can be solved using ADMM:
\begin{subequations}
\begin{align}
x_i^{k+1} &= \argmin_{x} f_{i}(x)+(x-y_i^k)^\tp z_i^k+\tfrac{\beta}{2} \norm{x-y_i^k}^2 \label{eq:admm1} \\
y_i^{k+1} &= \frac{1}{n} \sum_{j=1}^n x_j^{k+1} \label{eq:admm2} \\
z_i^{k+1} &= z_i^k + \beta\,(x_i^{k+1}-y_i^{k+1}) \label{eq:admm3}
\end{align}
\end{subequations}
where $(x_i^k,y_i^k,z_i^k)$ are the variables associated with agent $i$ at time $k$, and $\beta$ is the ADMM parameter. To implement this algorithm, however, each agent must solve the local optimization problem~\eqref{eq:admm1} \emph{exactly} as well as compute the \emph{exact} average~\eqref{eq:admm2} at each iteration. Instead, we consider a variant where the computations and communications are \emph{inexact}. Specifically, we replace the exact minimization~\eqref{eq:admm1} with a single gradient step with initial condition $y_i^k$ and stepsize $\alpha>0$, and we replace the exact averaging step~\eqref{eq:admm2} with a single gossip step using the  Laplacian matrix $\L^k$. This gives the following inexact version of ADMM:
\begin{align*}
x_i^{k+1} &= y_i^k - \alpha\,\bigl(\df(y_i^k)+z_i^k\bigr) \\
y_i^{k+1} &= x_i^{k+1} - \sum_{j=1}^n \L_{ij}^{k+1}\,x_j^{k+1} \\
z_i^{k+1} &= z_i^k + \beta\,(x_i^{k+1}-y_i^{k+1})
\end{align*}
Defining the state $w_i^k \defeq -\tfrac{\alpha}{\beta} z_i^{k-1}$, this algorithm is equivalent to Algorithm~\ref{alg:canonical} with $\gamma=1+\beta$ and $\delta=1$. In other words, SVL corresponds to an inexact version of ADMM, where $\alpha$ is the stepsize of the gradient step and $\beta$ is the ADMM parameter. See~\cite{DADMM1,DADMM2} for other distributed ADMM variants.

\begin{figure*}[thb]
	\centering
	\includegraphics[scale=0.90]{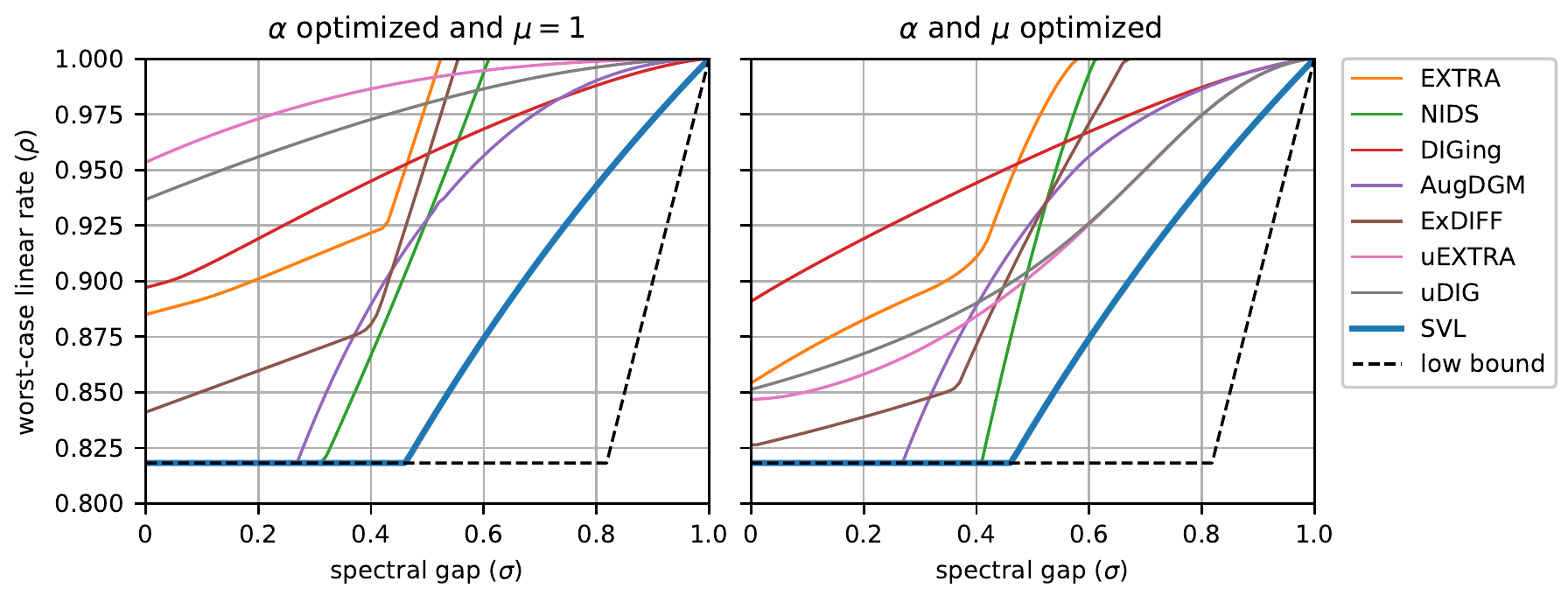}
	\caption{Comparison of upper bounds for linear convergence rate $\rho$ (smaller is better) as a function of graph connectedness $\sigma$, derived from Theorem~\ref{thm:main-result} using $\kappa=10$.
	(Left) stepsize $\alpha$ is optimized for each algorithm. (Right) both stepsize $\alpha$ and overelaxation parameter $\mu$ are optimized for each algorithm. The SVL algorithm (derived in Section~\ref{sec:algo_design}) outperforms all the tested methods. SVL has no tunable parameters so it is the same in both scenarios. The lower bound (see Section~\ref{sec:lower_bounds}) corresponds to $\rho \ge \tfrac{\kappa-1}{\kappa+1} \approx 0.818$ (optimal centralized gradient rate) and $\rho \ge \sigma$ (optimal average consensus rate).\label{fig:money_plot}}
\end{figure*}

\subsection{Special Cases}
\label{sec:special-cases}

\color{black}
We now show how the SVL algorithm reduces to well-known consensus and optimization algorithms in special cases.

\paragraph{$n=1$:} \blue{With} only \blue{one} agent, the distributed optimization problem~\eqref{eq:distrop-problem} is equivalent to centralized optimization. In this case, the Laplacian matrix is simply the scalar $\L^k=0$, so $v_1^k=0$ for all $k\ge 0$. Algorithm~\ref{alg:canonical} then simplifies to
\[
  x_1^{k+1} = x_1^k - \alpha\,\df(x_1^k), \qquad x_1^0\text{ arbitrary,}
\]
which is ordinary gradient descent with stepsize $\alpha$.  The fastest possible gradient rate of $\rho=\frac{\kappa-1}{\kappa+1}$ is achieved when $\alpha=\frac{2}{L+m}$.

\paragraph{$\kappa=1$:} When the condition ratio is unity (i.e., $m=L$), the distributed optimization problem~\eqref{eq:distrop-problem} is equivalent to average consensus. In this case, the parameters of SVL are simply $\alpha=\tfrac{1}{L}$, $\beta=1$, $\gamma=2$, and $\delta=1$. Also, the objective functions are quadratic, so we may assume without loss of generality that they have the form $\blue{f_i^k}(x) = \tfrac{L}{2} \|x-\blue{r_i^k}\|^2$, where~$\blue{r_i^k}\in\R^d$ is a parameter on agent $i\in\{1,\ldots,n\}$ \blue{at iteration $k$}. The SVL algorithm then simplifies to
\[
  x_i^{k+1} = x_i^k - \sum_{j=1}^n \L_{ij}^k\,x_j^k + \bigl(r_i^k-r_i^{k-1}\bigr), \qquad x_i^0 = r_i^0,
\]
which is a dynamic average consensus algorithm since the reference signals are continually injected into the dynamics~\cite{kia2019tutorial}. \blue{When the objective functions are constant, the $r_i$ terms cancel from the iterations and only affect the initial conditions.} \blue{This case} is \blue{referred to as \emph{static}} average consensus~\cite{tsitsiklis}, \blue{and} the worst-case rate of convergence is $\rho=\sigma$~\cite{xiaoboyd04}.


\section{Numerical Results}\label{numerics}
In this section, we compare the worst-case performance of SVL with \blue{that of} other first-order distributed algorithms.

\subsection{Algorithm Comparison (Upper Bounds)}
\label{sec:upper_bounds}
\blue{Theorem~\ref{thm:main-result} provides an upper bound on the worst-case convergence rate. We used this result to compare all algorithms}
in Table~\ref{table::algo_params}, \blue{including SVL}. The results are shown in Fig.~\ref{fig:money_plot}.
For each algorithm, we used a bisection search to find the smallest rate $\rho$ that yielded a feasible solution to the SDP~\eqref{eq:sdp-main}. We implemented the SDP in Julia~\cite{julia} with the JuMP~\cite{jump} modeling package and the Mosek \blue{interior point} solver~\cite{mosek}.
\blue{In an outer loop,} we performed a parameter search for each algorithm \blue{to find the step size $\alpha$ and overrelaxation parameter $\mu$} that yielded the smallest possible~$\rho$. Specifically, we used Brent's method and the Nelder--Mead method, respectively, as implemented in the Optim package~\cite{optim} as $\sigma$ ranged from 0 to 1.

As shown in Fig.~\ref{fig:money_plot}, \blue{optimizing over $\mu$ further improves worst-case performance}. Our proposed SVL algorithm outperforms all methods we tested.
\blue{Also shown in Fig.~\ref{fig:money_plot} is the lower bound described in Section~\ref{sec:lower_bounds}, namely $\rho \ge \max\{\tfrac{\kappa-1}{\kappa+1},\sigma\}$, which holds for any distributed algorithm.}

\begin{figure*}[ht]
	\centering
  \includegraphics[scale=0.90]{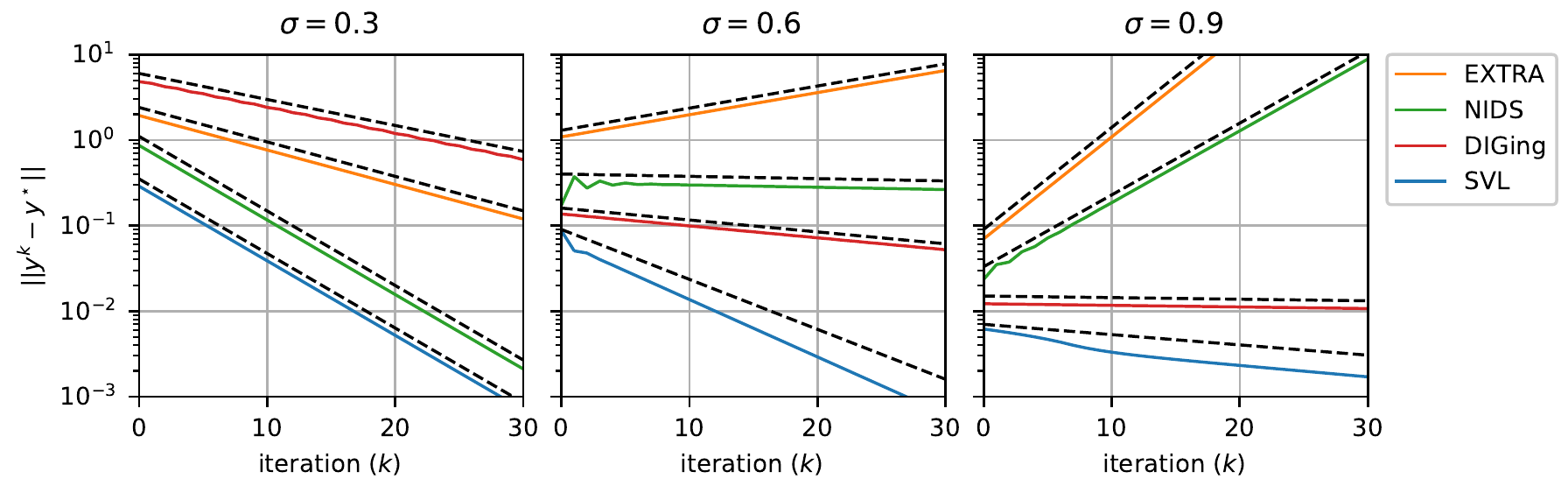}
	\caption{Approximate worst-case trajectories for EXTRA, NIDS, DIGing, and SVL. \blue{Trajectories were found by solving the relaxed problem~\eqref{opt:low_bound}}. We used $\alpha$ \blue{optimized} as in Fig.~\ref{fig:money_plot} and the default $\mu=1$. Simulations were performed for $\kappa=10$, $\sigma\in\{0.3, 0.6, 0.9\}$, and $n=d=2$. Dashed lines indicate corresponding upper bounds obtained from Theorem~\ref{thm:main-result} and shown in Fig.~\ref{fig:money_plot}. All traces were vertically translated to improve clarity.}
	\label{fig:wc-ex}
\end{figure*}

\subsection{Approximate Worst-Case Examples (Lower Bounds)}\label{sec:worst-case}

\blue{In an effort to show that the upper bounds for each algorithm in Fig~2 were likely tight, we searched for signals $\{x^k,u^k,v^k,y^k,z^k\}$ that satisfied~\eqref{alg} for some choice of $f_i$ and $\L^k$ satsifying Assumptions~\ref{assumption:local_functions} and \ref{assumption:Laplacian}, respectively.}

\blue{We first solved a relaxed version of the problem, where we replaced Assumptions~\ref{assumption:local_functions} and \ref{assumption:Laplacian} by the weaker conditions~\eqref{quad:func} and \eqref{quad:graph}, respectively. We used the following greedy heuristic. For a given algorithm and rate $\rho$, we solved~\eqref{eq:sdp-main} to obtain $(P,Q,R)$. At each time step $k$, we then maximized the Lyapunov increment $V^{k+1}-\rho^2 V^k$, where $V^k$ is defined in~\eqref{eq:V}.
We solved the following optimization problem for $k\ge 0$.}

\blue{\begin{equation}\label{opt:low_bound}
	\begin{aligned}
	\maximize_{u^k_i,v^k_i\in\R^d} \qquad &V^{k+1}-\rho^2 V^k \\
	\text{such that}\qquad & \text{\eqref{alg1}, \eqref{alg3}, and \eqref{quad:graph} hold,} \\
	& \text{\eqref{quad:func} holds for $i=1,\dots,n$,} \\
	& 1^\tp v^k = 0.
\end{aligned}
\end{equation}
For $k=0$, we also included $x^0$ as an optimization variable and the normalization $V^0=1$. For $k\ge 1$, we solved~\eqref{opt:low_bound} using the $x^k$ found at the previous iteration and warm-starting $u^k,v^k$. We used the Ipopt~\cite{Ipopt} local solver with default settings since~\eqref{opt:low_bound} is a nonconvex quadratically constrained quadratic program. Note that we must choose parameters $n$ and $d$.}

\blue{Our relaxed heuristic using $n=d=2$ was successful in constructing trajectories that matched the worst-case bounds from~\eqref{eq:sdp-main}. To illustrate, we simulated EXTRA, NIDS, DIGing, and SVL with $\kappa=10$ and a few values of $\sigma$ in Fig.~\ref{fig:wc-ex}. For each trajectory, we plotted $\|y^k-y^\star\|$ together with the corresponding upper bound $\rho$ found from Theorem~\ref{thm:main-result}. We obtained similar results for the other algorithms from Table~\ref{table::algo_params}.}

\blue{Since we used the relaxation~\eqref{quad:graph} to construct $z^k$ and $v^k$, there is no guarantee that there will exist a \emph{linear} Laplacian $\L^k$ such that $v^k = \L^k z^k$. However, finding whether such an $\L^k$ exists amounts to solving a convex optimization problem:
\begin{equation}\label{eq:constructL}
\begin{aligned}
\minimize_{\L^k \in \R^{n\times n}} \qquad & \norm{I-\Pi-\L^k} \\
\text{such that}\qquad &(\L^k \otimes I) z^k = v^k, \\
&\L^k 1 = 0,\quad 1^\tp \L^k = 0.
\end{aligned}
\end{equation}
If~\eqref{eq:constructL} is feasible and its optimal value is less than or equal to $\sigma$, then the associated $\L^k$ is a valid Laplacian matrix at timestep $k$. While there is no guarantee that \eqref{eq:constructL} will even be feasible, we reasoned that since there are $n^2$ variables and $2n+ndc$ linear constraints, where $d$ and $c$ are the number of rows of $C_y$ and $C_z$, respectively, we could increase our chances of finding feasible $\L^k$ with $n$ large and $d$ and $c$ small.}

\blue{In Figure~\ref{fig:NIDS-unstable}, we show a successful construction for the NIDS algorithm, which has $c=1$. We solved~\eqref{opt:low_bound} with $n=15$ and $d=1$, and solved~\eqref{eq:constructL} at each timestep. An optimal cost for~\eqref{eq:constructL} of $\sigma$ was always achieved.}
\blue{This result indicates that the upper bound for NIDS in Fig.~\ref{fig:money_plot} is likely tight, and that NIDS is not robustly stable in the time-varying setting. In other words, the \emph{network-independent} rate bound enjoyed by NIDS in the constant-graph setting~\cite[Thm.~2]{NIDS} does not carry over to the time-varying setting.}

\begin{figure}
	\centering
	\includegraphics[scale=0.90]{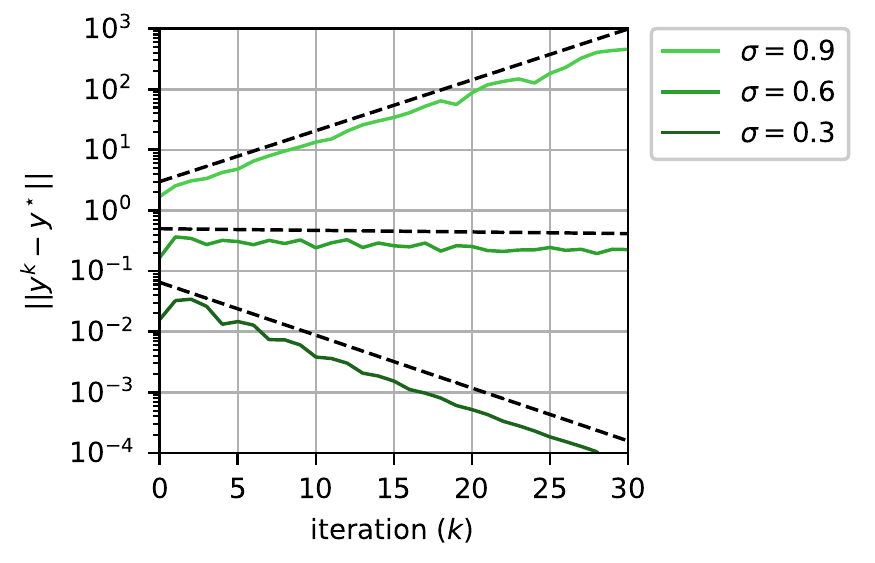}
	\caption{\blue{Worst-case trajectories for NIDS found by solving~\eqref{opt:low_bound} and successfully solving~\eqref{eq:constructL} to construct a sequence of Laplacians~$\{\L^k\}$. Simulations were performed using optimized $\alpha$, $\mu=1$, $\kappa=10$, $n=15$, and $d=1$ for $\sigma\in\{0.3,0.6,0.9\}$. Trajectories were plotted with their accompanying rate bounds (dashed lines) from Theorem~\ref{thm:main-result} and translated to improve clarity.}}
\label{fig:NIDS-unstable}
\end{figure}

\begin{rem}
	\blue{There may be other approaches to finding a worst-case $\L^k$ that perform better. For example, one might try alternating convex optimizations or including $\L^k$ directly as an optimization variable in a nonlinear program.} 
\end{rem}


\section{Conclusion}
\color{black}
We presented a universal analysis framework for a broad class of first-order distributed optimization algorithms over time-varying graphs.  The framework provides worst-case certificates of linear convergence via semidefinite programming, and we show empirically that our rate bounds are likely tight.  Optimizing the SDP from our analysis framework, we designed a novel distributed algorithm, SVL, which outperforms \blue{all known algorithms in this time-varying setting}.
\newpage


\begin{small}
\bibliographystyle{abbrv}
\bibliography{distralg_tcns}
\end{small}

\newpage
\appendix
\section{Appendix}

\subsection{Proof of Proposition~\ref{prop:fixedpoint}}\label{app:fixedpoint}

\blue{%
Suppose~\eqref{eq:fixedpoint} holds, and denote the optimizer of~\eqref{eq:distrop-problem} by $y_\text{opt}$. Then there exist vectors $p$ and $q$ such that
\begin{align*}
	\left\{\begin{aligned}
		0 &= (A-I)\,p \\
		y_\text{opt} &= C_y p \\
		0 &= F_x p
	\end{aligned}\right.
	\qquad\text{and}\qquad
	\left\{\begin{aligned}
		B_u &= (A-I)\,q \\
		D_{yu} &= C_y q \\
		D_{zu} &= C_z q.
	\end{aligned}\right.
\end{align*}
For all $i\in\{1,\ldots,n\}$, use these vectors to define the points
\begin{align*}
	x_i^\star &= p - q\,\df_i(y_\text{opt}), &
	y_i^\star &= y_\text{opt}, &
	z_i^\star &= C_z p, \\
	u_i^\star &= \df_i(y_\text{opt}), &
	v_i^\star &= 0.
\end{align*}
This is a fixed point of algorithm~\eqref{alg}, and the fixed point satisfies the conditions in~\eqref{eq:fixedpoint_conditions} since $y_\text{opt}$ is the optimizer of~\eqref{eq:distrop-problem}.

Now suppose $(x^\star,y^\star,z^\star,u^\star,v^\star)$ is a fixed point of~\eqref{alg} satisfying~\eqref{eq:fixedpoint_conditions}. Let $p = (1/n) \sum_{i=1}^n x_i^\star$. Since $1^\tp u^\star = 0$, $v^\star = 0$, and $1^\tp y^\star$ is unconstrained, we have from~\eqref{alg1} and~\eqref{alg3} that $p\ne 0$ is in the set~\eqref{eq:fixedpoint1}. Now let $v$ be any nonzero vector such that $v^\tp1=0$. Then from~\eqref{alg1}, we have that
\[
	0 = \bmat{A-I \\ C_y \\ C_z} (v^\tp x^\star) + \bmat{B_u \\ D_{yu} \\ D_{zu}} (v^\tp u^\star).
\]
Since this must hold for arbitrary $v^\tp u^\star$, this implies~\eqref{eq:fixedpoint2}.~\hfill\qed
}

\subsection{Proof of Theorem~\ref{thm:main-result}} \label{proof_main}

Assumptions~\ref{assumption:local_functions} and \ref{assumption:Laplacian} lead to quadratic inequalities that will be useful in proving our main result. These are stated in the following propositions.

\begin{prop}\label{prop:functions}
	Suppose Assumption~\ref{assumption:local_functions} holds for the local objective functions $f_i$. \blue{Let $(y_i^k,u_i^k)$ satisfy~\eqref{alg2}, and let $(y_i^\star,u_i^\star)$ be a fixed point that satisfies~\eqref{eq:fixedpoint_conditions}. Then
	\begin{equation}\label{quad:func}
		\bmat{\tilde y^k \\ \tilde u^k}^\tp (M_0\otimes I) \bmat{\tilde y^k \\ \tilde u^k} \ge 0.
	\end{equation}
	}
\end{prop}
\blue{%
\begin{proof}
Using the definition of $M_0$, the quadratic form is
\[
	\bmat{\tilde y^k \\ \tilde u^k}^\tp (M_0\otimes I) \bmat{\tilde y^k \\ \tilde u^k}
		= -2\sum_{i=1}^n (\tilde u_i^k - m\tilde y_i^k)^\tp (\tilde u_i^k - L\tilde y_i^k).
\]
Since the fixed point satisfies~\eqref{eq:fixedpoint_conditions}, Assumption~\ref{assumption:local_functions} implies that this is nonnegative with $y_\text{opt} = y_1^\star = \ldots = y_n^\star$.
\end{proof}
}

\begin{prop}\label{prop:graph}
	Suppose Assumption~\ref{assumption:Laplacian} holds for the graph $\mathcal{G}^k$ at each iteration. \blue{Let $(z_i^k,v_i^k)$ satisfy~\eqref{alg2}, and let $(z_i^\star,v_i^\star)$ be a fixed point that satisfies~\eqref{eq:fixedpoint_conditions}. Then for all $R\succeq 0$,
	\begin{equation}\label{quad:graph}
		\bmat{\tilde z^k \\ \tilde v^k}^\tp \bigl( M_1\otimes (I-\Pi)\otimes R\bigr) \bmat{\tilde z^k \\ \tilde v^k} \ge 0.
	\end{equation}
	}
\end{prop}
\begin{proof}
From the definition of the matrix norm and Assumption~\ref{assumption:Laplacian}, we have that
\begin{align*}
	\sigma \ge \normm{I-\Pi-\L^k}
	&= \normm{(I-\Pi-\L^k)(I-\Pi)} \\
	&= \max_{y\in \R^{n}, y\ne 0} \frac{\normm{(I-\Pi-\L^k)(I-\Pi)y}}{\norm{y}}.
\end{align*}
Without loss of generality, $y = \Pi \eta + (I-\Pi)\,\phi$, where $\eta$ and $\phi$ are arbitrary. By orthogonality, $\norm{y}^2 = \norm{\Pi \eta}^2 + \norm{(I-\Pi)\,\phi}^2$. Substituting the decomposition of $y$ into the above \blue{inequality},
\begin{align*}
	\sigma &\ge
	\max_{\phi,\eta\in\R^n, y\ne 0} \frac{\normm{(I-\Pi-\L^k)(I-\Pi)\,\phi}}{\sqrt{\norm{\Pi \eta}^2 + \norm{(I-\Pi)\,\phi}^2 }} \\
	&= \max_{\phi\in\R^n, y\ne 0} \frac{\normm{(I-\Pi-\L^k)(I-\Pi)\,\phi}}{\norm{(I-\Pi)\,\phi}} \\
	&= \max_{\phi\in\R^n, y\ne 0} \frac{\normm{(I-\Pi)(\phi-\L^k \phi)}}{\norm{(I-\Pi)\,\phi}},
\end{align*}
where the last two steps follow because the maximum is attained with $\eta=0$, and $\L^k \Pi = \Pi \L^k = \0$. Squaring both sides and rewriting as a quadratic form yields
\begin{align}\label{ineq_simple}
	\bmat{\phi \\ \L^k\phi}^\tp \bigl(M_1\otimes (I-\Pi)\bigr) \bmat{\phi \\ \L^k\phi} \ge 0
\end{align}
\blue{%
for all $\phi\in\R^n$. Now let $p$ denote the dimension of $z_i^k$. Then since $R\succeq 0$, it has the decomposition
\[
	R = \sum_{\ell=1}^p \mu_\ell \, w_\ell w_\ell^\tp,
\]
where $w_\ell\in\R^p$ and $\mu_\ell\ge 0$. Then using that $\tilde v^k = (\L^k\otimes I_p)\,\tilde z^k$, the quadratic form is
\begin{align*}
	\bmat{\tilde z^k \\ \tilde v^k}^\tp \bigl(M_1\otimes (I-\Pi)\otimes R\bigr) \bmat{\tilde z^k \\ \tilde v^k}
	&= \sum_\ell \mu_\ell \bmat{\star}^\tp \bigl(M_1\otimes (I-\Pi)\bigr) \bmat{(I\otimes w_\ell^\tp)\,\tilde z^k \\[1pt] (I\otimes w_\ell^\tp)\,\tilde v^k} \\
	&= \sum_\ell \mu_\ell \bmat{\star}^\tp \bigl(M_1\otimes (I-\Pi)\bigr) \bmat{(I\otimes w_\ell^\tp)\,\tilde z^k \\[1pt] \L^k\,(I\otimes w_\ell^\tp)\,\tilde z^k},
\end{align*}
which is nonnegative from~\eqref{ineq_simple} with $\phi \leftarrow (I\otimes w_\ell^\tp)\,\tilde z^k$.
}\end{proof}

\blue{%
Let $(x^k,y^k,z^k,u^k,v^k)$ denote a trajectory of algorithm~\eqref{alg}. Since the algorithm satisfies the fixed point conditions~\eqref{eq:fixedpoint} (by assumption), we have from Proposition~\ref{prop:fixedpoint} that there exists a fixed point $(x^\star,y^\star,z^\star,u^\star,v^\star)$ satisfying~\eqref{eq:fixedpoint_conditions}. The global optimizer is unique from Assumption~\ref{assumption:local_functions}, so the fixed point conditions~\eqref{eq:fixedpoint_consensus_optimality} imply that $y_1^\star = \ldots = y_n^\star = y_\text{opt}$ with $y_\text{opt}$ the optimizer of~\eqref{eq:distrop-problem}.

Since the trajectory satisfies the invariant~\eqref{alg3} and the columns of $\Psi$ form a basis for the nullspace of $\bmat{F_x & F_u}$, there exists a vector $\tilde s^k$ such that
\[
	\Psi\, \tilde s^k = \frac{1}{\sqrt{n}}\sum_{i=1}^n \bmat{\tilde x_i^k \\ \tilde u_i^k}.
\]
Multiplying the matrix in~\eqref{eq:lb-rho} on the right and left by $\tilde s^k$ and its transpose, respectively, we obtain the consensus inequality
\begin{subequations}\label{eq:inequalities}
	\begin{equation}\label{eq:cons_ineq}
		(\tilde x^{k+1})^\tp (\Pi\otimes P)\, \tilde x^{k+1}
		- \rho^2\, (\tilde x^k)^\tp (\Pi\otimes P)\, \tilde x^k \\
		+ \bmat{\tilde y^k \\ \tilde u^k}^\tp (M_0\otimes\Pi) \bmat{\tilde y^k \\ \tilde u^k} \leq 0.
	\end{equation}
	Now choose the vectors $w_2,\ldots,w_n\in\R^n$ such that the matrix $\bmat{1_n/\sqrt{n} & w_2 & \ldots & w_n}$ is orthonormal. Then we can multiply the matrix in~\eqref{eq:lmi-dis} on the right and left by the weighted sum
	\[
		\sum_{i=1}^n (w_\ell)_i \bmat{\tilde x_i^k \\ \tilde u_i^k \\ \tilde v_i^k}
	\]
	and its transpose, respectively, and sum over $\ell\in\{2,\ldots,n\}$ to obtain the disagreement inequality
	\begin{multline}\label{eq:dis_ineq}
	(\tilde x^{k+1})^\tp \bigl((I-\Pi)\otimes Q\bigr) \tilde x^{k+1} - \rho^2 (\tilde x^k)^\tp \bigl((I-\Pi)\otimes Q\bigr)\tilde x^k \\
	+ \bmat{\tilde y^k \\ \tilde u^k}^\tp \bigl(M_0\otimes (I-\Pi)\bigr) \bmat{\tilde y^k \\ \tilde u^k}
	+ \bmat{\tilde z^k \\ \tilde v^k}^\tp\bigl( M_1\otimes (I-\Pi)\otimes R\bigr) \bmat{\tilde z^k \\ \tilde v^k} \le 0,
	\end{multline}
\end{subequations}
where we used that $\{w_i\}_{i=1}^n$ form an orthonormal basis for $\R^n$. Summing the inequalities in~\eqref{eq:inequalities}, we obtain
\begin{equation*}
	V^{k+1} - \rho^2 V^k
	+ \bmat{\tilde y^k \\ \tilde u^k}^\tp (M_0\otimes I) \bmat{\tilde y^k \\ \tilde u^k}
	+ \bmat{\tilde z^k \\ \tilde v^k}^\tp \bigl( M_1\otimes (I-\Pi)\otimes R\bigr) \bmat{\tilde z^k \\ \tilde v^k} \le 0,
\end{equation*}
where $V^k$ is defined in~\eqref{eq:V}. The quadratic forms in the last two terms are nonnegative from Propositions~\ref{prop:functions} and~\ref{prop:graph}, which implies $V^{k+1}\le\rho^2\,V^k$. We then apply this inequality iteratively to obtain $V^k \le \rho^{2k}\,V_0$ for all $k\ge 0$. Now define
\[
	T \defeq \Pi\otimes P + (I-\Pi)\otimes Q,
\]
and note that $T\succ 0$ since $P$ and $Q$ are positive definite. Then letting $\cond(T) = \lambda_\text{max}(T)/\lambda_\text{min}(T)$ denote the condition number of $T$, we have the bound
\begin{align*}
	\|x_i^k\!-\!x_i^\star\|^2 &\le \|x^k\!-\!x^\star\|^2
	\le \cond(T)\,V^k
	\le \rho^{2k}\cond(T)\,V^0.
\end{align*}
Therefore, the bound~\eqref{eq:bound} holds with $c = \sqrt{\cond(T)\,V^0}$.~\hfill\qed
}

\subsection{Proof of Theorem~\ref{thm:SVL}}\label{sec:SVL-proof}

\blue{Substituting the template~\eqref{eq:SVL-template} into the LMI~\eqref{eq:lb-rho} reduces to
\[
  P_{11} \bmat{1-\rho^2 & -\alpha \\ -\alpha & \alpha^2} + M_0 \preceq 0,
\]
which is satisfied with $\alpha = (1-\rho)/m$ and $P_{11}=\tfrac{m\,(L-m)}{\rho\,(1-\rho)}$. Note that this LMI is known to describe the convergence rate of centralized gradient descent; see~\cite[Section 4.4]{lessard16}.}

Now consider the potential solution \blue{to~\eqref{eq:lmi-dis} given by}
\begin{align*}
  \blue{Q} &= \frac{t_3}{\alpha^2\rho^2} \bmat{1+\rho^2\frac{t_1}{t_4} & -1 \\ -1 & 1}
  \quad\text{and}\quad
  \blue{R} = \frac{t_5}{\alpha^2 t_2},\quad\text{where} \hspace{-5cm}\\
	t_1 &\defeq 2\,(1-\beta)-\eta, &
	t_2 &\defeq \beta-1+\rho^2, \\
	t_3 &\defeq \beta\,(\eta+2\rho^2)-\eta\,(1-\rho^2), &
	t_4 &\defeq 2\beta\rho^2 - \eta\,(1-\rho^2), \\
	t_5 &\defeq (1-\beta-\rho)\bigl(\beta\,(\eta+2\rho^2)-(1-\rho^2) (1-\kappa+2\kappa\rho)\bigr),\hspace{-5cm}\\
	t_6 &\defeq \bigl(2-\alpha(L+m)\bigr)(1\!-\!\rho^2)^2 \!-\! \bigl(2(1\!-\!\rho^4)-\alpha(L+m)\bigr)\beta.\hspace{-5cm}
\end{align*}
Using these values along with the value for $\sigma^2$ in~\eqref{eq:sigopt}, the matrix in~\eqref{eq:lmi-dis} is equal to the rank-one matrix $-\tfrac{1}{t_2 t_4} z z^\tp$, where
\begin{align*}
	z \defeq \frac{1}{\alpha\rho} \bmat{
		t_6 \\
		-t_2 t_3 \\
		\alpha\,t_2\,\bigl(2-\alpha\,(L+m)\bigr) \\
		\beta\,\bigl(t_3-\alpha\rho^2(L+m)\bigr)}.
\end{align*}
In order for this to be a valid solution, we must have $t_3 > 0$ and $t_1/t_4 > 0$ (so that $\blue{Q}\succ 0$), $t_5/t_2\ge 0$ (so that $\blue{R\succeq 0}$), and~$t_2 t_4>0$ (so that~\eqref{eq:lmi-dis} holds). All of these inequalities hold if and only if~\eqref{eq:beta-constraint} holds. Therefore, the SDP has a rank-one solution using the parameters in~\eqref{eq:paramopt} if $\beta$ and $\rho$ satisfy~\eqref{eq:sol-opt}. The convergence bound then follows from Theorem~\ref{thm:main-result} \blue{and Remark~\ref{rem:bounds}.}~\hfill\qed

\end{document}